\documentclass[12pt,a4paper,reqno]{amsart}

\usepackage[active]{srcltx}
\usepackage{a4wide}
\usepackage{amssymb, amsmath, mathtools}
\usepackage{graphicx}
\usepackage{xcolor}
\usepackage{float}
\usepackage[active]{srcltx}
\usepackage[ansinew]{inputenc}

\usepackage{hyperref}

\makeatletter
\@namedef{subjclassname@2020}{\textup{2020} Mathematics Subject Classification}
\makeatother

 \usepackage{graphics}

\theoremstyle{plain}% default
\newtheorem{theorem}{Theorem}[section]
\newtheorem{maintheorem}{Theorem}
\newtheorem{lemma}[theorem]{Lemma}
\newtheorem{proposition}[theorem]{Proposition}
\newtheorem{corollary}[theorem]{Corollary}

\theoremstyle{definition}

\theoremstyle{remark}
\newtheorem{remark}[theorem]{Remark}%[section]

\newcommand{\ols}[1]{\mskip.5\thinmuskip\overline{\mskip-.5\thinmuskip {#1} \mskip-.5\thinmuskip}\mskip.5\thinmuskip} % overline short
\newcommand{\olsi}[1]{\,\overline{\!{#1}}} % overline short italic
\makeatletter
\newcommand\closure[1]{
  \tctestifnum{\count@stringtoks{#1}>1} %checks if number of chars in arg > 1 (including '\')
  {\ols{#1}} %if arg is longer than just one char, e.g. \mathbb{Q}, \mathbb{F},...
  {\olsi{#1}} %if arg is just one char, e.g. K, L,...
}
\long\def\count@stringtoks#1{\tc@earg\count@toks{\string#1}}
\long\def\count@toks#1{\the\numexpr-1\count@@toks#1.\tc@endcnt}
\long\def\count@@toks#1#2\tc@endcnt{+1\tc@ifempty{#2}{\relax}{\count@@toks#2\tc@endcnt}}
\def\tc@ifempty#1{\tc@testxifx{\expandafter\relax\detokenize{#1}\relax}}
\long\def\tc@earg#1#2{\expandafter#1\expandafter{#2}}
\long\def\tctestifnum#1{\tctestifcon{\ifnum#1\relax}}
\long\def\tctestifcon#1{#1\expandafter\tc@exfirst\else\expandafter\tc@exsecond\fi}
\long\def\tc@testxifx{\tc@earg\tctestifx}
\long\def\tctestifx#1{\tctestifcon{\ifx#1}}
\long\def\tc@exfirst#1#2{#1}
\long\def\tc@exsecond#1#2{#2}
\makeatother

\def\R{\ensuremath{\mathbb R}}
\def\N{\ensuremath{\mathbb N}}

\def\emb{\operatorname{Emb}}
\def\inc{\operatorname{inc}}

\def\leb{\ensuremath{m}}

\newcommand{\qand}{\quad\text{and}\quad}

\numberwithin{equation}{section}

\begin{document}

\title[]{Impulsive Lorenz semiflows: Physical measures, \\statistical stability and entropy stability}
\author[J. F. Alves]{Jos\'{e} F. Alves}
\address{Jos\'{e} F. Alves\\ Departamento de Matem\'{a}tica\\ Faculdade de Ci\^encias da Universidade do Porto\\ Rua do Campo Alegre 687\\ 4169-007 Porto\\ Portugal}
\email{jfalves@fc.up.pt} \urladdr{http://www.fc.up.pt/cmup/jfalves}

\author[W. Bahsoun]{Wael Bahsoun}
\address{Wael Bahsoun, Department of Mathematical Sciences, Loughborough University, Loughborough, Leicestershire, LE11 3TU, UK}
\email{W.Bahsoun@lboro.ac.uk}
 \urladdr{https://www.lboro.ac.uk/departments/maths/staff/wael-bahsoun}

\date{\today}
\thanks{J. F. Alves (JFA) was partially supported by CMUP (UID/MAT/00144/2019) and
PTDC/MAT-PUR/4048/2021, which are funded
by FCT (Portugal) with national (MEC) and European structural funds
through the program  FEDER, under the partnership agreement PT2020. W. Bahsoun (WB) is supported by EPSRC grant EP/V053493/1. This work was initiated during JFA's Institute of Advanced Studies (IAS Loughborough) Residential Fellowship in July 2023. JFA and WB would like to thank CUMP and the Department of Mathematical Sciences at Loughborough for their hospitality during the course of this work. Both authors would like to thank Alexey Korepanov and Marcelo Viana for useful discussions.}


\begin{abstract}
We study semiflows generated via impulsive perturbations of Lorenz flows. We prove that such semiflows admit a finite number of physical measures. Moreover, if the impulsive perturbation is small enough, we show that the physical measures of the semiflows are close, in the weak* topology, to the unique physical measure of the Lorenz flow. A similar conclusion holds for the entropies associated with the physical measures.
\end{abstract}

\maketitle

\setcounter{tocdepth}{2}

\tableofcontents %\clearpage

\section{Introduction}
Lorenz  \cite{L63} studied numerically the three-dimensional vector field defined by 
\begin{equation}\label{eq.lorenzeq}
\left\{\begin{array}{l}
\dot{x_1}=\sigma(x_2-x_1),\\
\dot{x_2}=rx_1-x_2-x_1x_3,\\
\dot{x_3}=x_1x_2-bx_3,
\end{array}
\right.
\end{equation}
when $\sigma=10$, $r=28$ and $b=8/3$, as a simplified model for atmospheric convection. By numerical calculations, Lorenz observed that the associated flow has a compact trapping region with a chaotic attractor. Later, a rigorous mathematical framework of similar flows, called geometric Lorenz flows, was introduced in \cite{ABS77, GW79}. In \cite{T99}, Tucker provided a computer-assisted proof that the classical Lorenz attractor is indeed a geometric Lorenz attractor.  In particular, it is a singular-hyperbolic attractor~\cite{MPP}, namely a nontrivial robustly transitive attracting invariant set containing a singularity (equilibrium point).  Moreover, it admits a unique physical measure; see for example~\cite{APPV09}. Recently, many authors studied the stability of such physical measure and other statistical properties under small \emph{smooth} perturbations of Lorenz flows \cite{AS14, BMR20, BR19}. It is know that invariant stable foliations of Lorenz flows persist under small smooth perturbations; see for instance \cite{Bort}. In particular, the associated Poincar\'e map inherits an invariant $C^{1+\alpha}$,  $\alpha>0$, stable foliation~\cite{AM17} and the Poincar\'e map can be represented by a two dimensional \emph{skew product}, where the base map is a one-dimensional piecewise uniformly expanding and the fibre maps are uniformly contracting. Such properties were exploited in   \cite{AS14, BMR20, BR19}.

Impulsive perturbations are very common in physical models. They capture phenomena characterised by sudden changes in the states of the system (e.g. when billiard balls collide, there is an impulse exchanged between them, resulting in abrupt changes in their velocities and directions).~In this work, we are interested in studying impulsive perturbations of Lorenz flows, where the resulting perturbations generate discontinuous impulsive semiflows (see Subsection \ref{subec:impluse} for a precise definition of an impulsive semiflow). Our motivation is twofold, on the one hand is to investigate stability of statistical properties for Lorenz flows under discontinuous perturbations, and on the other hand to investigate smooth ergodic theoretic properties in the context of impulsive dynamical systems; in particular, existence and stability of \emph{physical measures}. The latter systems have been mainly studied from a qualitative point of view: existence and uniqueness of solutions, sufficient conditions to ensure a complete characterization and some asymptotic stability of the limit sets \cite{B07,BF07,BF08,BS19,C04a,C04b,K94a}, with few exceptions on basic ergodic theoretic results \cite{ABJ22, ACS17, ACV15}. 

Our work demonstrates how introducing a small impulsive perturbation to Lorenz flows can pose non-obvious problems, requiring techniques different from those traditionally employed in classical scenarios.  At the technical level, the main difficultly in our study is that the perturbed semiflow is not continuous. Hence, in general, it is not expected that it admits a smooth invariant foliation. Consequently, to prove existence of physical measures for the semiflow through a corresponding Poincar\'e map, one cannot use the standard technique of quotienting along stable leaves, since the Poincar\'e map is \emph{not} necessarily a skew product. We believe that the current work will stimulate further research on statistical properties for impulsive semiflows.

 In the rest of this introduction, we recall the definition of a physical measure and the notion of an impulsive dynamical system. Statements of our main results (Theorem \ref{th.main} and Theorem \ref{th.stability}) on existence and stability of physical measures and the corresponding entropies for impulsive Lorenz semiflows are found in Subsection \ref{se.Lorenz}.
 
 \subsection{Physical measures and $u$-Gibbs measures}
 Let $M$ be a   finite dimensional compact Riemannian manifold (possibly with  boundary) and   $m$ be the Lebesgue (volume) measure   on the Borel sets of $M$. We say that $X: \R^+_0 \times M\to M$
 is a   \emph{semiflow}  if, for all $x\in M$ and    $s,t\in \R^+_0$, we have
\begin{enumerate}
\item $X_0(x)= x$,
\item $X_{t+s}(x)=X_t(X_s(x))$.
\end{enumerate}
The \emph{trajectory} of  $x\in M$ is the curve defined  by $X_t(x)$, for $t\ge0$.
The semiflow is called a \emph{flow} when we have $\R$ playing the role of $\R^+_0$, which means that we can consider the past of the trajectories.

Given a semiflow $X$ on $M$, we define the \emph{basin} $\mathcal B_\nu$ of  a  probability measure $\nu$ on the Borel sets of $M$ as the  set of points 
 $x \in M$ such that, for any continuous 
  $\varphi:M \to \mathbb R$, we have
\begin{equation}\label{eq.basina}
    \lim_{T\rightarrow\infty}\frac{1}{T}\int_{0}^{T}\varphi(X_t(x))dt= \int\varphi\, d\nu .
\end{equation}
We say that  $\nu$ is a \emph{physical measure} for $X$ if $m(\mathcal B_\nu)>0$.   
%Note that a physical measure for $X$  is  necessarily   \emph{$X$-invariant}, that is,  $\mu(X_t^{-1}(A))=\mu(A)$, for every $t\ge 0$ and every Borel subset $A\subset  M$.  
Since measures defined on the Borel sets of a compact metric space are determined by their integrals on continuous functions, it follows that distinct physical measures must have disjoint basins. 

Physical measures for semiflows are frequently obtained considering discrete-time dynamical systems associated with these flows, e.g. via Poincar\'e return maps. In the discrete-time case of a dynamical system defined by a map $f:M\to M$, we define the basin $\mathcal B_\mu$ of a  probability measure $\mu$ on the Borel sets of $M$ as the set of points $x \in M$ such that, for any continuous 
  $\varphi:M \to \mathbb R$, we have
\begin{equation}\label{eq.basinb}
    \lim_{n\rightarrow\infty}\frac{1}{n}\sum_{j=0}^{n-1}\varphi(f^j(x))dt= \int\varphi\, d\mu .
\end{equation}
We say that  $\mu$ is a \emph{physical measure} for $f$ if $m(\mathcal B_\mu)>0$. 
Note that a physical measure    is  necessarily   \emph{$f$-invariant}. 
%, i.e.  ${\mu(f^{-1}(A))=\mu(A)}$, for any Borel subset $A\subset  M$.  

A  probability measure $\mu$ on the Borel sets of $M$ is called a \emph{$u$-Gibbs measure} (also known as an \emph{SRB measure}) for a map ${f:M\to M}$  if \emph{i)} $\mu$ is $f$-invariant, \emph{ii)} $f$ has positive Lyapunov exponents $\mu$ almost everywhere, and \emph{iii)} the conditional measures of $\mu$ on local unstable manifolds are absolutely continuous with respect to the Lebesgue measures on these manifolds. 
A result due to Pesin~\cite{P76}
establishes  that if $f$ is a    $C^{1+\alpha}$  diffeomorphism     whose Lyapunov exponents are all nonzero with respect to an ergodic invariant probability measure $\mu$, then the \emph{stable holonomy} is absolutely continuous.
Using   Birkhoff Ergodic
Theorem, it can be proved   that
the basin of an ergodic probability measure~$\mu$ in a compact Riemannian manifold~$M$ contains $\mu$ almost every point in~$M$; see e.g. \cite[Proposition~2.12]{A20}.
Since  time averages with respect to continuous functions are constant on stable manifolds, it follows that every ergodic $u$-Gibbs probability measure for a $C^{1+\alpha}$ diffeomorphism with non-zero Lyapunov exponents $\mu$ almost everywhere  is a physical measure. 
These conclusions, relatively well known for diffeomorphisms of class $C^{1+\alpha}$, are also valid for the class of piecewise $C^{1+\alpha}$ diffeomorphims considered in \cite{S92,S92a} which we will address in this work; see also \cite[Part I \& Part II]{KSL86}.

The existence of $u$-Gibbs measure for piecewise piecewise hyperbolic diffeomorphisms has been deduced in several situations under additional assumptions:
\begin{itemize}
\item bounded derivative \cite{DL08, C99b,Y85,Y98};
\item  conservative setting (billiards) \cite{CZ05,CZ05a,Y98};
\item invariant stable foliation \cite{APPV09, GL20};
\item stable cone field transverse to the singularity set \cite{DZ14};
\end{itemize}
None of these assumptions are a priori verified for the Poincar\'e maps we consider. The most suitable work for our setting is that of Pesin \cite{P92} and Satayev \cite{S92,S92a} on   piecewise hyperbolic  diffeomorphisms.

\subsection{Impulsive semiflows}\label{subec:impluse} 
Consider  a semiflow or a flow $X$   on $M$ and   $\Sigma\subset M$  a region such that
$$\inf\left\{t> 0:X_t(x)\in \Sigma\right\}>0,\quad \forall x\in M.$$
Given a map 
 $\varphi: \Sigma\to M$   and $x\in M$, we define the \emph{impulsive trajectory} ${\gamma_x:[0,T(x))\to M}$, for some $T(x)\in\R^+ \cup\{\infty\}$,
  and  the  \emph{impulsive times} $0=\tau_0(x)<\tau_1(x)<\tau_2(x)<\cdots$  (a~finite or  infinite number of times)  using the following inductive procedure:
\begin{enumerate}
\item Set  $$
\tau_1(x)=
\begin{dcases}
\inf\left\{t> 0:X_t(x)\in \Sigma\right\} ,& \text{if } X_t(x)\in \Sigma\text{ for some }t>0;\\
+\infty, & \text{otherwise;}
\end{dcases}
$$
and   for $0\le t<\tau_{1}(x)$
$$\gamma_x(t)=X_t(x).$$
 \begin{figure}[h]
 \begin{center}
 \includegraphics[width=8cm]{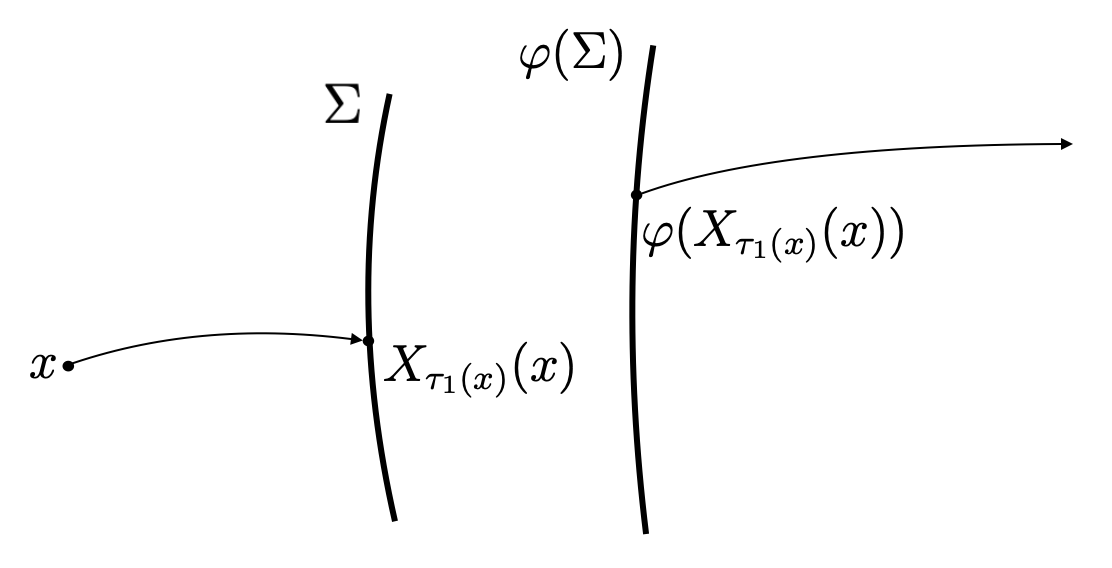}
 \caption{Impulsive trajectory}
 \end{center}
 \end{figure}

%We proceed inductively.
%$$\gamma_x(t)=X_t(x),\quad \text{for $0\le t<\tau_{1}(x)$},\quad  \gamma_x(\tau_1(x))=I(X_{\tau_1(x)}(x)),$$
%and proceed inductively.
%\begin{enumerate}
%\item firstly we set
%  $$\gamma_x(\tau_1(x))=I(X_{\tau_1(x)}(x)).$$
%  Defining the second impulsive time of $x$ as
%  $$\tau_2(x)=\tau_1(x)+\tau_1(\gamma_x(\tau_1(x))),$$
% we set
%         $$\gamma_x(t)=X_{t-\tau_1(x)}(\gamma_x(\tau_1(x))),\quad\text{for }\tau_1(x)<t<\tau_2(x).$$
%         \item 
%         
\item Assume  that  $\tau_n(x) $ and $\gamma_x(t)$ have been  defined   for    $0\le t< \tau_{n}(x)$ and $n\ge 1$. \\
If $\tau_{n}(x)=\infty$,  then set $T(x)=\infty$  and stop the process.\\
If $\tau_{n}(x)<\infty$, set
 $$\quad\quad\gamma_x(\tau_{n}(x))=\varphi(X_{\tau_n(x)-\tau_{n-1}(x)}(\gamma_x({\tau_{n-1}(x)}))),  $$
$$
\tau_{n+1}(x)=\tau_n(x)+\tau_1(\gamma_x(\tau_n(x)))
$$
and  for $\tau_n(x)<t<\tau_{n+1}(x)$
$$\gamma_x(t)=X_{t-\tau_n(x)}(\gamma_x(\tau_n(x))).$$
This completes the inductive procedure.
      \end{enumerate}
      
In  case $\tau_n(x)<\infty$ for all $n\ge1$, set  %the time length of the trajectory of $x$ as
 $$T(x)=\sup_{n\ge 1}\,\{\tau_n(x)\}.
 $$
In general  we can have  $T(x)=+\infty$ or $T(x)<+\infty$. However,    assuming that $\varphi(\Sigma)\cap\Sigma=\emptyset$, it follows that   $T(x)=+\infty$ for all $x\in M$; see \cite[Remark 1.1]{AC14}.
%In this work  we are going to consider only the first possibility, meaning that  the impulsive trajectories of points in $X$ are defined for all $t\ge0$.
%
%\begin{remark}\label{re.infty} Under the fairly reasonable condition $I(D)\cap (D)=\emptyset$, for instance, we have  $T(x)=\infty$ for all $x\in X$. Indeed, it has been proved in \cite[Theorem~2.7]{C04a} that $\tau_1$ is always lower semicontinuous on the set $X\setminus D$. As $D$ is compact and $I$ is continuous, then $I(D)$ is compact. Supposing that $I(D) \cap D=\emptyset$, then by the lower semicontinuity of $\tau_1$ on $X\setminus D$  there must be some $\alpha>0$ such that $\tau_1(x)>\alpha$ for all $x\in I(D).$ This clearly implies that $T(x)=\infty$ for all $x\in X$.
%\end{remark}
We say that $(M,X,\Sigma,\varphi)$ is an \emph{impulsive dynamical system} if
 $$\tau_1(x)>0\quad\text{and}\quad T(x)=+\infty, \quad\text{for all $x\in M$}.$$
 %We call $\Sigma$ the \emph{impulsive set}, $I$ the \emph{impulsive function} and $\tau_1$ the \emph{first impulsive time} of the impulsive dynamical system.
 We call  $\Sigma$ the \emph{impulsive region} and $\varphi$ the \emph{impulsive map} of the impulsive dynamical system. % $(M,X,\Sigma, I)$.
The  \emph{impulsive semiflow} $Y:\R^+_0 \times M\to M$ of the impulsive dynamical system $(M,X, \Sigma, \varphi)$ is given by
$$
Y (t,x) = \gamma_x(t),\quad \forall (x,t)\in\R^+_0\times M,
         $$
where $\gamma_x(t)$ is the impulsive trajectory of $x$ determined by $(M,X,\Sigma, \varphi)$.
It follows from  \cite[Proposition 2.1]{B07} that $Y$ is indeed a semiflow, though not necessarily continuous nor a flow, even if $X$ is a flow. In fact, through the points of $\varphi(\Sigma)$ more than one trajectory can pass, which means that we cannot consider the past of these trajectories. A different type of impulsive perturbations was considered in \cite{GV20}.

\subsection{Lorenz flows}\label{se.Lorenz} 
Our main results will be obtained not only for the  flow $X$ associated with the vector field in~\eqref{eq.lorenzeq}, but more generally for the broader class of geometric Lorenz flows introduced in \cite{ABS77, GW79}. We will refer to all these flows as \emph{Lorenz flows}.
Many features of the flow $X$ in the trapping region $M$ have been  proven  over the last decades by studying   the Poincar\'e return map to the plane $x_3=r-1$; see  \cite{ABS77,ABS82}.
As noted in \cite[Section~2.2]{S92} and \cite[Section 2]{S92a} this  return map  fits in a certain set of conditions (L1)-(L4) that we describe in Subsection~\ref{Lorenzmaps}. In particular, using an appropriate coordinate system,  the Poincar\'e return map  is given by  a transformation  $F:\Sigma\to \Sigma$, where
 $$\Sigma=\{(x_1,x_2, x_3)\in\R^3: |x_1|\le 1,\,|x_2|\le 1, x_3=r-1\},$$
 and has a discontinuity line $\Gamma\subset \Sigma$ given by  $x_1=0$, dividing  $\Sigma$  into the two subdomains  
such that  $F$ is smooth in both  subdomains. Strictly speaking, the transformation is not defined on the set $\Gamma$, but this is completely irrelevant for our purposes. It is well known that the return time $R:\Sigma\to \R^+$ at a point $x=(x_0,x_1, x_3)\in\Sigma$ satisfies
\begin{equation}\label{eq.are}
R(x) \approx |\log |x_0||.
\end{equation}
The return map $F$ has been exhaustively studied by several authors from various perspectives. In particular, it is well known that it admits a stable invariant foliation by nearly vertical lines, which makes it possible to obtain some of its properties studying the one-dimensional transformation in the quotient space  by the stable leaves. This strategy was particularly successful in obtaining a $u$-Gibbs measure $\mu$ for $F$, which incidentally proves to be the unique physical measure for $F$, with its basin covering Lebesgue almost all of~$\Sigma$.  In the case of geometric Lorenz flows, the Poincar\'e return map  $F$ and the return time $R$ enjoy these same properties by construction.

In this work we consider impulsive semiflows associated with Lorenz flows, having an impulsive map $\varphi:\Sigma\to M$
close to  the inclusion map, with the impulsive region $\Sigma$ being the domain of a Poincar\'e return map as above. Allowing the impulsive map $\varphi$ to send points \emph{above}  $\Sigma$, it  is relatively easy to think of examples for which the impulsive flow becomes essentially trivial. Indeed, consider a small $t_0>0$ and the impulsive map $\varphi:\Sigma\to M$ defined for each $x\in\Sigma$ by 
$\varphi(x)=X_{-t_0}(x).$
Since the trajectories under a Lorenz flow of Lebesgue almost all points in $M$ hit the region $\Sigma$, it follows  that the trajectories of Lebesgue  almost all  points in~$M$ enter a region where the flow is made up of periodic trajectories.
We therefore consider the impulsive maps having  images in a   flow box \emph{below}~$\Sigma$, defined for some small $t_0>0$ as
$$B^+(\Sigma)=\{ X_t(z): 0< t\le t_0,\; z\in\Sigma\}.$$
Let the inclusion map of $\Sigma$ in $M$ be denoted by $\inc_\Sigma$,  and    the space of $C^1$ embeddings of $\Sigma$ in $M$ be denoted by $\emb^1(\Sigma,M)$. Our first main result is the following theorem:

%We shall refer to a flow $X$ as above or geometric Lorenz flows as introduced in \cite{ABS77, GW79} simply as a Lorenz flow.

\begin{maintheorem}\label{th.main}
Let $X$ be a Lorenz flow and $(M,X,\Sigma,\varphi)$ be an impulsive dynamical system such that 
  $\varphi(\Sigma)\subset B^+(\Sigma).$
If  $\varphi$ is a $C^2$ map sufficiently close to   $\inc_\Sigma$ in $ \emb^1(\Sigma,M)$, then the semiflow associated with $(M,X,\Sigma,\varphi)$ has a  finite number of physical measures whose basins cover Lebesgue almost all of $M$.
\end{maintheorem}

An interesting open problem is to investigate if the impulsive semiflow has a unique physical measure.~Note that the uniqueness of the physical measure for Lorenz flows results from the uniqueness of the measure for the quotient transformation (because of transitivity), an instrument that we cannot use in our case.~Uniqueness for the impulsive Lorenz semiflow will naturally follow from the uniqueness of the physical measures given by Corollary~\ref{co.finite2} for the Poincar\'e return map, something that is not guaranteed by the results in~\cite{S92,S92a}.

A question that is of particular interest is whether the Lorenz flow is stable under small impulsive perturbations.
We say that the Lorenz flow $X$ is \emph{statistically stable by impulsive perturbations} on $\Sigma$ if, for any sequence $(\varphi_n)_n$ converging to $\inc_\Sigma$ in the $C^0$ topology and any sequence $(\nu_n)_n$, where  $\nu_n$  is a  physical measure for the impulsive flow of $(M,X,\Sigma,\varphi_n)$, 
we have $\nu_n\to \nu$ in the weak* topology, as $n\to\infty$, where $\nu$ is the physical measure for~$X$. We say that $X$ is \emph{entropy stable by impulsive perturbations} on $\Sigma$ if $h_{\nu_n}(Y_n)\to h_{\nu}(X)$, as $n\to\infty$, where each $Y_n$ is the impulsive semiflow of $(M,X,\Sigma,\varphi_n)$ and $h_{(\cdot)}(\cdot)$ denotes the respective metric entropy. Our second main result is the following theorem:

\begin{maintheorem}\label{th.stability} 
 Lorenz flows are statistically stable and entropy stable by impulsive perturbations on $\Sigma$.
\end{maintheorem}

\section{Piecewise hyperbolic diffeomorphisms} \label{se.phm}
This section is devoted to the presentation of some results in \cite{S92,S92a} for piecewise hyperbolic diffeomorphisms that will play an important role in this work. 
Let $\Sigma$ be a compact  finite dimensional  Riemannian manifold, possibly with a boundary. Denote by~$\rho$ the distance in $\Sigma$ and by $m$ the Lebesgue measure on $\Sigma$. We say that  $f:\Sigma\rightarrow\Sigma$ 
is  a \emph{piecewise   diffeomorphism} if there exists a finite number of  pairwise disjoint open regions $D_1,\dots,D_q$ such that  $\Sigma=\closure{\cup_{i=1}^q {{D}}_i}$
and  $f\vert_{\cup_i D_i}$ is an injective  differentiable  map. 
Set 
$$D=\cup_{i=1}^q D_i\qand \Gamma=\Sigma\setminus D.$$
We shall refer to $\Gamma$ as the \textit{singularity set} of $f$ and assume $m(\Gamma)=0$.
The derivative of $f$ at a  point $x\in D$ will be denoted by  $d_xf: T_x\Sigma\to  T_{f(x)}\Sigma$.

\subsection{Physical and $u$-Gibbs measures} 
The results on the existence of $u$-Gibbs measures  for piecewise differentiable  diffeomorphisms are obtained  in \cite{S92,S92a}  under some hyperbolicity assumptions that will be stated below. First of all, we  assume the existence of an \emph{unstable confield} $(\mathcal K_x^u)_{x\in D}$ and a \emph{stable conefield} $(\mathcal K_x^s)_{x\in D}$; see \cite[Section~1.3]{S92} for precise definitions. We say that a submanifold $V\subset D$ is a \emph{u-disk} (resp.  \emph{s-dis}k) of size $r>0$ if it is a ball of radius $r$ in its intrinsic metric and the tangent space $T_x V$ is contained in $K_x^u$ (resp. $K_x^s$) for every $x\in V$. The Lebesgue measure on a $u$-disk   $V$ will be denoted by $m_V$.

\begin{enumerate}

\item[(H1)] There are constants $A,\alpha>0$ such that, for any $x\in D$,
$$\|d_xf\|\le A \rho(x,\Gamma)^{-\alpha}\qand \|d^2_xf\|\le A \rho(x,\Gamma)^{-\alpha}.
$$

\item[(H2)] There are conefields $(\mathcal K_x^u)_{x\in D}$ and $(\mathcal K_x^s)_{x\in D}$ such that, for any $x\in D$,
$$d_xf(\mathcal K_x^u)\subset \mathcal K_{f(x)}^u\qand (d_xf)^{-1}(\mathcal K_{f(x)}^s)\subset \mathcal K_{x}^s.
$$
Moreover, there exists $\lambda>1$ such that for any vectors $v\in \mathcal K_{x}^u$ and $w\in \mathcal K_{f(x)}^s$
$$
\|d_xf(v)\|\ge\lambda \|v\|\qand \|(d_xf)^{-1}(w)\|\ge\lambda \|w\|.
$$

\item[(H3)] There are constants $B,\beta >0$ such that, for any $n\in\N$ and $\varepsilon>0$,
$$m(f^{-n}(\Gamma_\varepsilon))<B\varepsilon^\beta.
$$

\item[(H4)] There is $\varepsilon_0>0$ such that, for any $u$-disk $V$ there are   $n_0=n_0(V)$ and $B_0=B_0(V)$ such that for any $\varepsilon\in(0,\varepsilon_0)$,
$$m_V(V\cap f^{-n}(\Gamma_\varepsilon))<\varepsilon^\beta m_V(V),\quad\forall n\ge n_0,
$$
$$
m_V(V\cap f^{-n}(\Gamma_\varepsilon))<B_0\varepsilon^\beta m_V(V),\quad\forall n\in\N.
$$

\item[(H5)] 
The partition $D_1,\dots,D_q$ is \emph{generating}:
for any $\varepsilon>0$, there is $n=n(\varepsilon)$ such that for any $j_{-n},\dots,j_n\in\{1,\dots,q\}$ the diameter of $\cap_{i=-n}^{n} f^{i}(D_{j_i})$ is smaller than~$\varepsilon$.

 \item[(H6)] There exist $\lambda>1$ and  $r_0>0$ such that
 \begin{enumerate}
\item for any points $x,y$ in the same $s$-disk of size $r_0$ contained in $f(D_j)$
 $$\rho(f^{-1}(x),f^{-1}(y))>\lambda \rho(x,y).$$
 \item for any points $x,y$ in the same $u$-disk of size $r_0$ 
 $$\rho(f(x),f(y))>\lambda \rho(x,y).$$
\end{enumerate}

 \item[(H7)] The constants $n_0(V)$ and $B_0(V)$ depend only on the size of $V$.

\end{enumerate}

Conditions (H1)-(H5) are enough  for ensuring the existence of $u$-Gibbs measures, the remaining  (H6)-(H7) will be used to deduce the continuity of the $u$-Gibbs measures in the weak* topology under some  additional assumptions to be presented in  Subsection~\ref{se.stat}. 

\begin{theorem}\label{th.finite}
If  $f$ satisfies (H1)-(H4), then there is a finite number of physical   measures $\mu^1,\dots, \mu^s$  for $f$  with $\leb(\Sigma\setminus (\mathcal B_{\mu^1}\cup\cdots \mathcal B_{\mu^s}))=0$ such that, for each $1\le i \le s$, we have
\begin{enumerate}
\item  $\mu^i$ is an ergodic $u$-Gibbs measure;
\item the entropy formula holds for   $\mu^i$;
\item there exists $A_i$ with $\mu^i(A_i)=1$ such that Lebesgue almost every point in $\mathcal B_{\mu^i}$ belongs in the stable manifold of some point in   $A_i$;
\item the densities   of the conditionals of   $\mu^i$ with respect to the Lebesgue measure on local unstable manifolds are  bounded from above and below by uniform constants.
\end{enumerate}
\end{theorem}

The conclusions of this theorem were essentially  all obtained  in \cite[Theorem 2 and 3]{P92} with a countable number of measures.
The above formulation with reduction to a finite number of measures follows from \cite[Theorems 5.14 and 5.15]{S92}, with the exception of the last item, which is not explicitly stated in \cite{S92}, but still follows from the results therein. In fact, consider   the set $D^-$ of points whose negative trajectories are defined and  do not hit the singularity set~$\Gamma$, i.e.
 $$D^-=\bigcap_{n=1}^\infty f^n(D).$$
Set for $r>0$ and $\chi\in(0,1)$
  \begin{align*}
M^-(r,\chi)&=\left\{x\in D^-:\rho(f^{-n}(x),\Gamma)\ge r\chi^n,\;\text{for all $n\ge0$}\right\}.
%\\
%M^+(r,\chi)&=\left\{x\in D^+:\rho(f^{n}(x),\Gamma)\ge r\chi^n,\;\text{for all $n\ge0$}\right\},
\end{align*}
For each $x\in M^-(r,\chi)$, there exists a local unstable manifold $V=V^u_{loc}(x,\delta(r))$ and, moreover,
the  conditional measure of each $\mu=\mu^i$ on   $V $ is given as 
 $$\mu_V(E)=\int_{V\cap E}p(y)dm_V(y),
 $$
where $m_V$ is Lebesgue measure on $V$. 
In addition, for all $y\in V$, 
\begin{equation}\label{eq.density1}
p(y)=\frac1{\int_V q (y,z) dm_V} q (y,z),
\end{equation}
and it follows from the Corollary after \cite[Lemma 5.13]{S92} that there are constants $\gamma>\alpha$ and  $C>0$ such that 
 \begin{equation}\label{eq.density2}
|q (y,z)-1|<C r^{\gamma-\alpha}.
\end{equation}
This gives that last item of Theorem~\ref{th.finite}.

\begin{remark}\label{re.uniform}
The estimates  in \cite{S92} give that  the constants $C$ and $\gamma$ in \eqref{eq.density2} only depend on the constants involved in (H1)-(H4).
\end{remark}

\subsection{Lorenz maps}\label{Lorenzmaps}
We present here a set of conditions given in \cite[Section 2]{S92a}, defining a class of piecewise hyperbolic  diffeomorphisms that satisfy conditions (H1)-(H7) and includes  the two-dimensional Poincar\'e return map $F:\Sigma\to\Sigma$ associated with the Lorenz flow in an appropriate  coordinate system.

\begin{enumerate}
\item[(L1)]    The domain $\Sigma$ is the square $\{(x,y): |x|\le 1,\,|y|\le 1\}\subset \R^2$
  and the curve $\Gamma\subset \Sigma$, given by  $x=0$, divides  $\Sigma$  into two subdomains $$D_1=\{(x,y)\in\Sigma: x>0\}\qand D_2=\{(x,y)\in\Sigma: x<0\}.$$
  \item[(L2)]  The map $f$ is smooth in both  domains $D_1$ and $D_2$.
   \item[(L3)]  The map $f$ is given by $f(x,y)=(G(x,y), H(x,y))$, with
   $$\|H_y\|<1\qand \|G^{-1}_x\|<1,$$
$$1-\|H_y\|\cdot \|G^{-1}_x\|>2\sqrt{\|G^{-1}_x\|\cdot \|H_y\|\cdot \|G^{-1}_x H_x\|},$$
$$\|H_y\|\cdot \|G^{-1}_x H_x\|\cdot\|G_y\|<(1-\|H_y\|)(1-\|G^{-1}_x\|),$$
%  \begin{enumerate}
%\item $\|H_y\|<1$ and $\|G^{-1}_x\|<1$;
%\item $1-\|H_y\|\cdot \|G^{-1}_x\|>2\sqrt{\|G^{-1}_x\|\cdot \|H_y\|\cdot \|G^{-1}_x H_x\|}$;
%\item $\|H_y\|\cdot \|G^{-1}_x H_x\|\cdot\|G_y\|<(1-\|H_y\|)(1-\|G^{-1}_x\|)$.
%\end{enumerate}
%
where $G_x, G_y,H_x $ and $H_y$  denote the partial derivatives 
%$\partial G/\partial x, \partial H/\partial y,\dots$ 
and $\|(x,y)\|=\sup\{|x|,|y|\}$.
\item[(L4)] Letting $  G_j, H_j$ denote respectively the restrictions of $ G,H$ to $D_j$, for $j=1,2$,   in  a neighbourhood of $\Gamma$ the functions $G_j(x,y)$, $H_j(x,y)$ have the form 
 \begin{align*}
G_j(x,y)&=\tilde G_j (y|x|^\beta,|x|^\alpha ),\\
H_j(x,y)&=\tilde H_j(y|x|^\beta,|x|^\alpha),
\end{align*}
with $\tilde G_j,\tilde H_j$   smooth  in a neighbourhood of $(0,0)$ and constants $\alpha\in(0,1)$ and $\beta>0$.
\end{enumerate}

\begin{theorem}\label{th.L->H}
If  $f$ satisfies (L1)-(L4), then $f$  satisfies (H1)-(H7). \end{theorem}

See \cite[Section 2 \& Section 3]{S92a} or the remark after \cite[Theorem 2.1]{S92}.

\subsection{Statistical stability}\label{se.stat}

Consider  a sequence of piecewise hyperbolic diffeomorphisms  $(f_n)_n$ and $f$ such that  (H1)-(H7) hold for  all $f_n$ and  $f$. For the conclusions about the continuity of   $u$-Gibbs measures we need the following  conditions. 

\begin{enumerate}
\item[(S1)] The constants $A,\alpha,B,\beta,\lambda,\varepsilon_0, B_0$ and the cones $\mathcal K^u_x, \mathcal K^s_x$ appearing in (H1)-(H4) are the same for all $f$ and $f_n$.
 \item[(S2)] The functions $f_j$ and $f_{n,j}$ have continuous extensions to the sets $D_j$ and $D_{n,j}$, respectively, for all $n$ and $j$.
  \item[(S3)] The sets $D_{n,j}$ converge to the sets  $D_j$, in the sense that for every $\varepsilon>0$ there is $n(\varepsilon)\in\N$ such that if $n>n(\varepsilon)$, then
  $$
  D_j\setminus \Gamma_\varepsilon\subset D_{n,j}\qand 
  D_{n,j} \setminus \Gamma_{n,\varepsilon}\subset   D_j.
  $$
  \item[(S4)] The restriction of $f_n$ to  $D_{n,j}$ is an equicontinuous family and it converges to the restriction of $f$ to $D_j$ in the $C^0$ topology.
   \item[(S5)] On every  $D_{j}\setminus \Gamma_\varepsilon$ the sequence $f_n$  converges to   $f$   in the $C^1$ topology.
  \end{enumerate}
  
  The next result is obtained in   \cite[Theorem 7.2]{S92}.
  
  \begin{theorem}\label{th.stable}
Consider   a sequence of maps $(f_n)_n$ and $f$  such that   conditions  (H1)-(H7) and (S1)-(S5) hold for all $f_n$ and $f$. If $\mu_n$ is a $u$-Gibbs measure for $f_n$ and $\mu$ is a limit point of $\mu_n$ in the weak* topology, then $\mu$ is a $u$-Gibbs measure for $f$. \end{theorem}

As a consequence of  Theorem~\ref{th.finite} and Theorem~\ref{th.stable}, we obtain the next result.
  
    \begin{corollary}\label{co.stable}
Consider   a sequence of maps $(f_n)_n$ and $f$  such that  conditions  (L1)-(L4) and (S1)-(S5) hold for all $f_n$ and $f$. If $\mu_n$ is a $u$-Gibbs measure for $f_n$ and $\mu$ is a limit point of $\mu_n$ in the weak* topology, then $\mu$ is a $u$-Gibbs measure for $f$. \end{corollary}

\section{Impulsive  Lorenz semiflows}

Consider  $X$   the Lorenz flow defined in an open set $M\subset \R^3$ and $F:\Sigma\to \Sigma$   the Poincar\'e return map as described in
Subsection~\ref{se.Lorenz}. Let   $\mu$ be the unique $u$-Gibbs measure for $F$. It is well known that the conditional measures of $\mu$ on local unstable manifolds have densities (with respect to Lebesgue measure) bounded from above and below by positive constants. Together with~\eqref{eq.are}, this gives that the return time $R$ is integrable with respect to $\mu$.  For a continuous function $\phi:M\to \R$, it is  well known that the measure given by
\begin{equation}\label{eq.ali0}
\nu(\phi)=\frac1{\int_{\Sigma} R d\mu }\int_{\Sigma'}\int_0^{R (x)}\phi\circ X_t(x) dt d\mu (x)
\end{equation}
is invariant under the Lorenz flow $X$; see Proposition~\ref{pr.invariance} below for a proof in the case   of impulsive perturbations of the Lorenz system (which contains the Lorenz system as the particular case $\varphi=\inc_\Sigma$). 
The entropy of the flow $X$ with respect to the measure $\nu$ is by definition the entropy of its time one map, i.e.
 $$h_{\nu}(X)=h_{\nu}(X_1)$$
By Abramov formula \cite{A59}, we have 
\begin{equation}\label{eq.abramov}
h_{\nu}(X)=\frac{h_{\mu}(F)}{\int_{\Sigma} R d\mu } .
\end{equation}  
Consider from now on the impulsive dynamical system $(M,X,\Sigma,\varphi)$ as in Theorem~\ref{th.main} and the corresponding impulsive flow $Y$. 

%\subsection{Poincar\'e map for the Lorenz flow}\label{se.Pmap}

\subsection{Poincar\'e return map}\label{sub.poincare}

%Recall that $\Sigma$ is a  global cross section for the flow $X$. 
%Let  $\Sigma'=\varphi(\Sigma)$. 
In this subsection we build a Poincar\'e return map for the impulsive semiflow $Y$.
Assuming $\varphi$ is sufficiently close to $\inc_\Sigma\in \emb^1(\Sigma,M)$, then $\Sigma'=\varphi(\Sigma)$ is also a global cross section for $X$.
We are going to define a Poincar\'e map for   $Y$ on $\Sigma'$. 
Since we assume $\Sigma'\subset B^+(\Sigma)$, we have   $\Sigma'$   below $\Sigma$, with respect to the $z$ coordinate. For each $x\in\Sigma'$, set  
  \begin{figure}[h]
 \begin{center}
 \includegraphics[width=8cm]{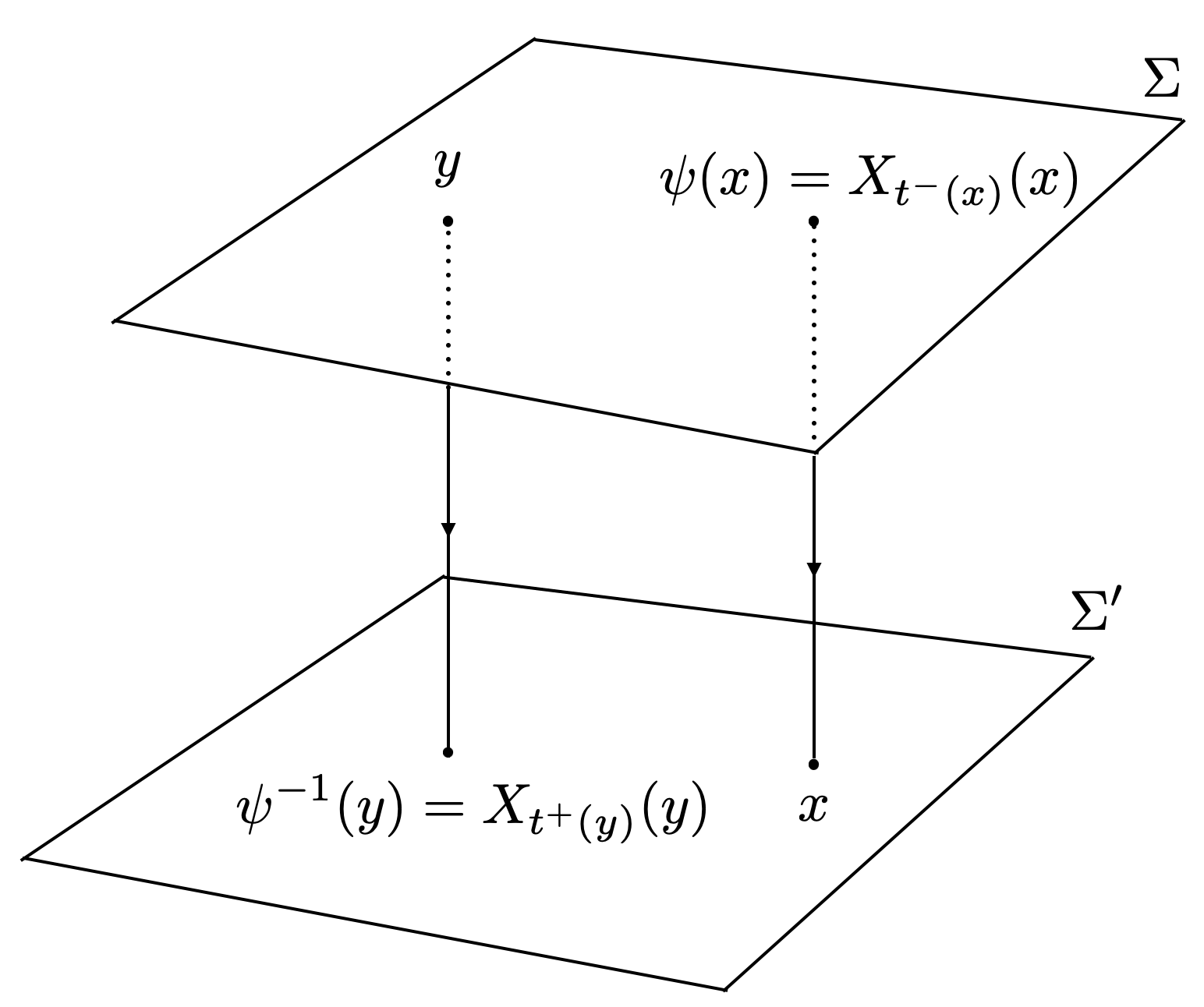}
 \caption{The maps $\psi$ and $\psi^{-1}$}
 \end{center}
 \end{figure}
 \begin{equation}\label{eq.t-}
t^-(x)=\max\{t<0 : X_{t}(x)\in \Sigma\}
\end{equation}
 and for $y\in\Sigma$, set  
  \begin{equation}\label{eq.t+}
  t^+(y)=\min\{t>0 : X_{t}(y)\in \Sigma'\}.
  \end{equation}
Let
 $\psi:\Sigma'\to\psi(\Sigma')\subset \Sigma$ be  the diffeomorphism given by 
 \begin{equation}\label{eq.psi}
\psi(x)= X_{t^-(x)}(x),
\end{equation}
 with an inverse $\psi^{-1}: \psi(\Sigma')\to\Sigma'$ given by
 $$\psi^{-1}(x)= X_{t^+(x)}(x).$$ 
The maps $t^\pm$, the section  $\Sigma'$ and the  diffeomorphism $\psi$  depend on the impulsive system, but for simplicity we will not make  this explicit  in the notation at this stage, see Figure 2 for a pictorial illustration.
 We define the \emph{Poincar\'e map} associated with the impulsive flow $F_Y:\Sigma'\to\Sigma'$ by
  \begin{equation}\label{eq.leo}
F_Y=\varphi\circ F\circ \psi .
\end{equation}

\begin{figure}[h]
 \begin{center}
 \includegraphics[width=10cm]{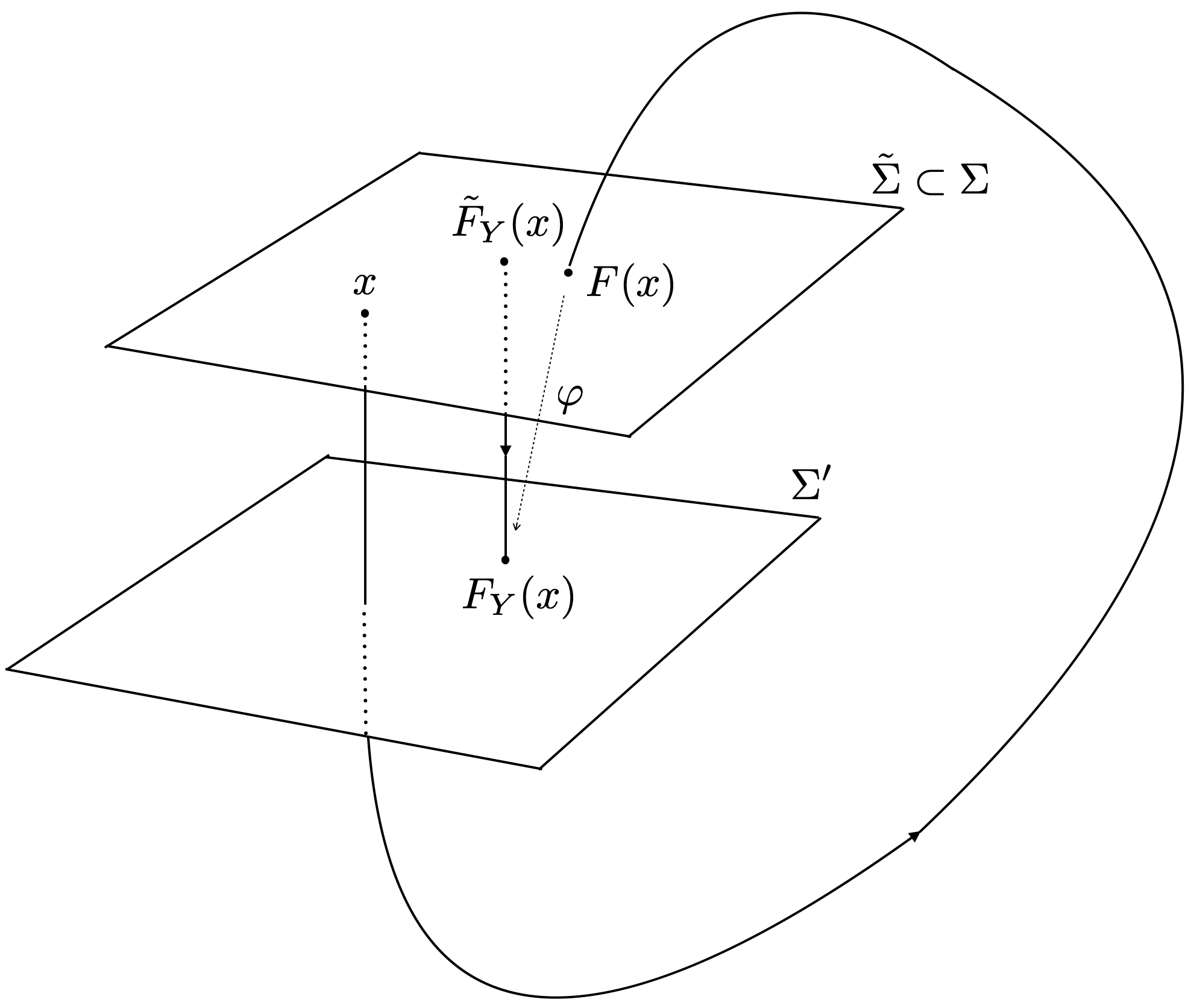}
 \caption{The maps $F_Y$ and $\tilde F_Y$}
 \end{center}
 \end{figure}

 Observe that $F_Y$ is a piecewise $C^{1+\alpha}$ map with a singularity set $\Gamma_Y=\psi^{-1}(\Gamma)$.
 Actually, we are going to   use  a version of $F_Y$   translated to $\tilde\Sigma=\psi(\Sigma')\subset \Sigma$, by
considering
%$$\tilde\Sigma=\psi(\Sigma')\subset \Sigma$$ 
% and  
 $\tilde F_Y:\tilde\Sigma\to\tilde\Sigma$  defined by
\begin{equation}\label{eq.lavinia}
 \tilde F_Y = \psi \circ F_Y \circ \psi^{-1} .
\end{equation}
Notice that $F_Y$ and $\tilde F_Y$ are conjugate dynamical systems via a smooth conjugacy, see Figure 3 for a pictorial illustration.
It follows from~\eqref{eq.leo} and \eqref{eq.lavinia} that
\begin{equation}\label{eq.hadi}
 \tilde F_Y= \psi\circ\varphi\circ F.
\end{equation}
Clearly,  $\tilde F_Y$ has the same singularity set $\Gamma$ as $F$.
Since  $\varphi$ is a $C^2$ map  close to $\inc_\Sigma$ in $ \emb^1(\Sigma,M)$, it follows  that  $\tilde F_Y$ is a small $C^{1+\alpha}$ perturbation of $F$. Notice however that $\tilde F_Y$ is \emph{not} necessarily a Poincar\'e map for a Lorenz flow near  $X$.

\begin{remark}
Although $F$ has an invariant (stable) foliation, there is no guarantee that the same happens with $F_Y$ (and therefore also $\tilde F_Y$). The absence of such an invariant foliation and the consequent lack of existence of a quotient transformation, as in the Poincar\'e transformation of the Lorenz flow, make the study of the ergodic properties of $\tilde F_Y$ considerably more complicated than that of $F$.
\end{remark}

\begin{lemma}\label{le.L1to5}
$\tilde F_Y$ satisfies (L1)-(L4).
\end{lemma}
\begin{proof} It follows from~~\eqref{eq.hadi} that  $\tilde F_Y$ has the same singularity set $\Gamma$ of $F$. 
Checking (L1)-(L2) is therefore just a matter of choosing the appropriate coordinate system. For (L3)-(L4), note  that $\varphi$ (and hence $\psi$) can be taken arbitrarily close to the inclusion map and  recall that $F$ satisfies (L3)-(L4); note also that the conditions in (L3) hold for small perturbations of $F$  and use \eqref{eq.hadi} to obtain  (L4).
\end{proof}

The previous lemma together with  Theorem~\ref{th.finite} yield the following conclusion.

\begin{corollary}\label{co.finite1}
If $\varphi$ is a $C^2$ map  close to $\inc_\Sigma$ in $ \emb^1(\Sigma,M)$, 
 then $\tilde F_Y$ has a finite number of physical   measures $\tilde\mu^1_Y,\dots, \tilde\mu^s_Y$      with $\leb(\tilde\Sigma\setminus (\mathcal B_{\tilde\mu_Y^1}\cup\cdots \mathcal B_{\tilde\mu_Y^s}))=0$ such that, for each $1\le i \le s$,  
\begin{enumerate}
\item  $\tilde\mu_Y^i$ is an ergodic $u$-Gibbs measure;
\item the entropy formula holds for   $\tilde\mu_Y^i$;
%\item $\tilde\mu_Y^i$ almost every point in $\mathcal B_{\tilde\mu_Y^i}$ belongs in the stable manifold of some point in the support of $\mu^i$;
\item there exists $A_i$ with $\tilde\mu^i_Y(A_i)=1$ such that Lebesgue almost every point in $\mathcal B_{\tilde\mu^i_Y}$ belongs in the stable manifold of some point in   $A_i$;
\item the densities   of the conditionals of   $\tilde\mu_Y^i$ with respect to the Lebesgue measure on local unstable manifolds are  bounded from above and below by uniform constants.
\end{enumerate}
%
%
%
%then $\tilde F_Y$   has a finite number of $u$-Gibbs measures $\tilde\mu_{Y}^1,\dots, \tilde\mu_{Y}^s$ whose basins cover Lebesgue almost all of $\tilde \Sigma$. Moreover, 
%\begin{enumerate}
%\item the entropy formula holds for each $\mu_Y^i$;
%\item the densities of the conditionals of each $\mu_Y^i$ with respect to Lebesgue measure   on local unstable manifolds are  bounded from above and below by positive uniform constants.
%\end{enumerate}
\end{corollary}

Since~\eqref{eq.hadi} gives that $F_Y$ and $\tilde F_Y$ are conjugate dynamical systems via the smooth conjugacy $\psi$, we get the next consequence of Corollary~\ref{co.finite1}.

\begin{corollary}\label{co.finite2}
If $\varphi$ is a $C^2$ map  close to $\inc_\Sigma$ in $ \emb^1(\Sigma,M)$,  then $F_Y$ has   physical measures $\mu_Y^1=\psi^{-1}_* \mu_Y^1,\dots,\mu_Y^s=\psi^{-1}_* \mu_Y^s$    
 with $\leb( \Sigma'\setminus (\mathcal B_{ \mu_Y^1}\cup\cdots \mathcal B_{ \mu_Y^s}))=0$ such that, for each $1\le i \le s$,  
%
%
%$F_Y$ has a finite number of $u$-Gibbs measures $\mu_{Y}^1,\dots, \mu_{Y}^s$ whose basins cover Lebesgue almost all of $\Sigma'$, with $\mu_Y^i=\psi^{-1}_*\tilde\mu_Y^i$, for $i=1,\dots p$. Moreover, 
\begin{enumerate}
\item  $ \mu_Y^i$ is an ergodic $u$-Gibbs measure;
\item the entropy formula holds for   $ \mu_Y^i$;
%\item Lebesgue almost all points in $\mathcal B_{ \mu_Y^i}$ belongs in the stable manifold of some point in the support of $\mu^i$;
\item there exists $A_i$ with $ \mu^i_Y(A_i)=1$ such that Lebesgue almost every point in $\mathcal B_{\mu^i_Y}$ belongs in the stable manifold of some point in   $A_i$;
\item the densities   of the conditionals of   $ \mu_Y^i$ with respect to the Lebesgue measure on local unstable manifolds are  bounded from above and below by uniform constants.
\end{enumerate}
%\begin{enumerate}
%\item the entropy formula holds for each $\mu^i$;
%\item the densities   of the conditionals of each $\mu^i$ with respect to Lebesgue measure on local unstable manifolds are  bounded from above and below by positive uniform constants.
%\end{enumerate}
\end{corollary}

\begin{remark} Even though the Poincar\'e map $F$ for the Lorenz flow is transitive in the attractor (and thus it has a unique $u$-Gibbs measure) and $F_Y$ is a small perturbation of $F$, we cannot infer the transitivity of $F_Y$. This would be true if $F_Y$ were the Poincar\'e map of a Lorenz-like flow.
\end{remark}

Note that the return time function $R_Y:\Sigma'\to\R$ associated with the Poincar\'e map $F_Y$ for the impulsive semiflow $Y$ is related to the return time $R$ of the Lorenz flow $X$     by
\begin{equation}\label{eq.rushes}
R_Y = R\circ\psi +t^- .
\end{equation}
In the following we simply use $\tilde\mu_Y$ to denote any of the measures given by Corollary~\ref{co.finite1}.

\begin{lemma}\label{lem01}
$R\in L^1(\tilde\mu_Y)$.
\end{lemma}

\begin{proof}
Corollary~\ref{co.finite1} gives that the densities of the conditional measures  of  $\mu_Y$ with respect to Lebesgue measure  on local unstable manifolds are  bounded from above and below by uniform constants.
Therefore, the integrability of $R$ with respect to $\tilde\mu_Y$ holds if   it holds with respect to Lebesgue measure on $\Sigma$. 
The integrability with respect to Lebesgue measure is a consequence of the fact that $R$ satisfies~\eqref{eq.are}.
\end{proof}

 \begin{corollary}\label{cor01}
$R_Y\in L^1(\mu_Y)$.
\end{corollary}

\begin{proof}
It follows from~\eqref{eq.rushes} that
$$
\int_{\Sigma'_Y} R_{Y}d\mu_Y= \int_{\Sigma'_Y} (R\circ\psi +t^-)d\psi^{-1}_*\tilde\mu_Y 
=
 \int_{\tilde\Sigma_Y} (R +t^-\circ\psi^{-1} ) d\tilde\mu_Y.
$$
Since $t^-\circ\psi^{-1}$ is a bounded function, the conclusion follows from Lemma~\ref{lem01}.
\end{proof}

\subsection{Physical measures}\label{se.physical}
In this subsection  we lift each physical  measure of the Poincar\'e map $F_Y:\Sigma'\to\Sigma'$ to a physical measure of the impulsive semiflow~$Y$. We perform the construction  using standard ideas on suspension flows, taking into account the difficulties due the the lack of continuity of the impulsive semiflow. 

First of all, we define a class of functions that will help us in defining the lifting. Consider~$\mathcal F$  the class of bounded functions $\phi: M\to \R$ such that,
for all $x\in\tilde \Sigma$, the function
 $$\R\ni t\longmapsto \phi \circ Y_t(x)\in \R
$$
has at most countably many discontinuity points. Clearly, $\mathcal F$ contains   all the  compositions $\phi\circ Y_s$, with $s\in\R^+_0$ and $\phi$ continuous, which obviously includes the continuous functions.
Given   an $F_Y$-invariant probability measure $\mu_{Y}$,  set  for $\phi\in \mathcal F$,
\begin{equation}\label{eq.ali}
\nu_Y(\phi)=\frac1{\int_{\Sigma'} R_Yd\mu_{ Y}}\int_{\Sigma'}\int_0^{R_Y(x)}\phi\circ Y_t(x) dt d\mu_{Y}(x),
\end{equation}
where $dt$ stands for the  integral with respect to the Lebesgue measure on $\R$,  which is obviously well defined  for all $\phi\in\mathcal F$. 

\begin{proposition}\label{pr.invariance}
$\nu_Y$ %in~\eqref{eq.ali} 
defines a probability measure  on the Borel sets of $M$ which is  invariant under the semiflow $Y$.
\end{proposition}
\begin{proof}
It is straightforward to check that $\nu_Y$ is a linear positive functional on the set of continuous functions and that $\nu_Y(1)=1$. It follows from Riesz-Markov Theorem that~$\nu_Y$ defines a probability measure   on the Borel sets of $M$. The proof of the invariance of $\nu_Y$ follows essentially as in the case of a continuous flow, we include it here for completeness. We need to check that
\begin{equation*}
\int \phi\circ Y_s d\nu_Y= \int \phi\circ d\nu_Y
\end{equation*}
for all continuous $\phi$.
First of all, notice that the integral on the left hand side is well defined, since $\phi\circ Y_s\in \mathcal F$. Furthermore, setting $C=\left(\int_{\Sigma'} R_Yd\mu_{Y}\right)^{-1}$, we may write
\begin{align*}
\int \phi\circ Y_s d\nu_Y &= C \int_{\Sigma'}\int_0^{R_Y(x)}\phi\circ Y_s\circ Y_t(x) dt d\mu_{Y}(x)\\
 & =  C \int_{\Sigma'}\int_0^{R_Y(x)}\phi\circ Y_{s+t}(x) dt d\mu_{Y}(x)\\
 & =  C \int_{\Sigma'}\int_s^{R_Y(x)+s}\phi\circ Y_{t}(x) dt d\mu_{Y}(x)\\
 & =  C \int_{\Sigma'}\left(\int_0^{R_Y(x)}\!\!\!\phi\circ Y_{t}(x) dt - \int_0^{s}\phi\circ Y_{t}(x) dt+\int_{R_Y(x)}^{R_Y(x)+s}\!\!\!\phi\circ Y_{t}(x) dt\right) d\mu_{Y}(x).
\end{align*}
We are left to show that
\begin{equation}\label{eq.kids}
 \int_{\Sigma'}\int_0^{s}\phi\circ Y_{t}(x) dt d\mu_{Y}(x)= \int_{\Sigma'}\int_{R_Y(x)}^{R_Y(x)+s}\phi\circ Y_{t}(x) dtd\mu_{Y}(x).
\end{equation}
Indeed, using the $F_Y$-invariance of  $\mu_{Y}$ we get
\begin{align*}
\int_{\Sigma'} \int_0^{s}\phi\circ Y_{t}(x) dt d\mu_{Y}(x) &= \int_0^{s} \int_{\Sigma'} \phi\circ Y_{t}(x)  d\mu_{Y}(x) dt\\
& =  \int_0^{s} \int_{\Sigma'} \phi\circ Y_{t}\circ F_Y(x)  d\mu_{Y}(x) dt \\
& =   \int_{\Sigma'} \int_0^{s} \phi\circ Y_{t}\circ F_Y(x)   dt d\mu_{Y}(x) \\
& =   \int_{\Sigma'} \int_0^{s} \phi\circ Y_{t}\circ Y_{R_Y(x)}(x)   dt d\mu_{Y}(x) \\
& = \int_{\Sigma'}\int_{R_Y(x)}^{R_Y(x)+s}\phi\circ Y_{t}(x) dtd\mu_{Y}(x).
\end{align*}
This gives~\eqref{eq.kids}, which concludes the proof.
\end{proof}

Let now $\mu_Y^1,\dots,\mu_Y^s$ be the physical measures  for the Poincar\'e return map $F_Y$ given by Corollary~\ref{co.finite2}. Consider the probability measures $\nu_Y^1,\dots,\nu_Y^s$ defined on the Borel sets of~$M$, where each $\nu_Y^i$ is related to $\mu_Y^i$ through the formula in~\eqref{eq.ali}.
With the next proposition we conclude  the proof of Theorem~\ref{th.main}. 

\begin{proposition}\label{pr.finite2}
$\nu_{Y}^1,\dots,\nu_{Y}^s$ are physical measures for $Y$ whose basins cover Lebesgue almost all of $M$.
\end{proposition}

\begin{proof}
It is well known that with the exception of trajectories contained  in the stable manifold of the singularity point 0, all other trajectories of points in the trapping region~$M$ hit the Poincar\'e section $\Sigma$. This implies  that the trajectories of Lebesgue almost all points in  $M$ must intersect $\Sigma$. Taking into account the definition of the Poincar\'e map $F_Y$ for the impulsive semiflow $Y$, we easily get that the trajectories of Lebesgue almost all points in~$M$ pass through $\Sigma'$. 
Since the basin of a measure is invariant under the dynamics, it  is enough to show that Lebesgue almost all points in $\Sigma'$ belong in the basin of one of the measures~$\nu_{Y}^i$, for some $i=1,\dots,s$.

By Corollary~\ref{co.finite2} we know that Lebesgue almost all points in $\Sigma'$ belong in the basin of one of the ergodic  measures $\mu_{Y}^i$, for some $i=1,\dots,s$.
Moreover,   there exists a  set~$A_i$ with $\tilde\mu^i_Y(A_i)=1$ such that Lebesgue almost every point in $\mathcal B_{\tilde\mu^i_Y}$ belongs in the stable manifold of some point in   $A_i$. Given a continuous function $\varphi: M\to\R$, consider $\varphi: \Sigma'\to\R$ given for $x\in\Sigma'$ by
 $$\hat\varphi(x)=\int_0^{R_Y(x)}\varphi(Y_t(x)) dt.
 $$
By Corollary~\ref{cor01} and Birkhoff's Ergodic Theorem we may assume that for $\mu_Y^i$   almost every $x\in A_i$ we have
\begin{equation}\label{eq.birkhoffR}
    \lim_{n\rightarrow\infty}\frac{1}{n}\sum_{j=0}^{n-1} R_Y(F_Y^j(x))dt= \int_{\Sigma'} R_Y\, d\mu_Y^i 
\end{equation}
and
\begin{equation}\label{eq.birkhoffhat}
       \lim_{n\rightarrow\infty}\frac{1}{n}\sum_{j=0}^{n-1} \hat\varphi(F_Y^j(x))dt= \int_{\Sigma'} \hat\varphi\, d\mu_Y^i.
\end{equation}
Without loss of generality, we can also assume that 
\begin{equation}\label{eq.finiteR}
R_Y(F_Y^n(x))<\infty,\quad\text{for all $x\in A_i$ and  $n\ge0$.}
\end{equation}
Given $x\in A_i$, set $T_0(x)=0$ and 
$$
 T_{n}(x)=R_Y(x)+ R_Y(F_Y(x))+\cdots+R_Y(F^{n-1 }(x)),\quad\text{for $n\ge 1$.}
$$
Given $T>0$, consider $n=n(T)\in\N$ such that $T_{n-1}(x)< T\le T_{n}(x)$. By \eqref{eq.finiteR},   such an integer $n$ always exists. Moreover,
\begin{equation}\label{eq.Tn}
T\to\infty \implies n\to\infty .
\end{equation}
We may write
\begin{align}
\frac1T\int_0^T\varphi(Y_t(x)) dt&= 
\frac1T\sum_{j=0}^{n-1}\int_{T_j(x)}^{T_{j+1}(x)}\!\!\!\varphi(Y_t(x)) dt
-
\frac1T \int_{T}^{T_{n }(x)}\!\!\!\varphi(Y_t(x)) dt\nonumber\\
&= 
\frac1T\sum_{j=0}^{n-1}\int_{0}^{R_Y(F_Y^j(x))}\!\!\!\varphi(Y_t(F_Y^j(x))) dt
-
\frac1T \int_{0}^{T_{n }(x)-T}\!\!\!\varphi(Y_t(F^n(x))) dt.\label{eq.finale}
\end{align}
Using  the definition of $\hat\varphi$, we get
 \begin{equation}\label{eq.chain}
\frac1T\sum_{j=0}^{n-1}\int_{0}^{R_Y(F_Y^j(x))}\varphi(Y_t(F_Y^j(x))) dt =
\frac1T\sum_{j=0}^{n-1}\hat\varphi( F_Y^j(x) )  =\frac{n}T\cdot\frac1n\sum_{j=0}^{n-1}\hat\varphi( F_Y^j(x) ) .
\end{equation}
Now observe that 
$$
\frac{T_{n-1}(x)}n<\frac Tn\le \frac{T_{n}(x)}n,
$$
which together with~\eqref{eq.birkhoffR} and \eqref{eq.Tn} yields 
\begin{equation}\label{eq.recall}
\lim_{T\to\infty}\frac Tn =\frac1{\int_{\Sigma'} R_Y\, d\mu_Y^i }.
\end{equation}
Since the measure $\nu_Y^i$ is related to $\mu_Y^i$ through the formula in~\eqref{eq.ali}, it follows from~\eqref{eq.birkhoffhat} and \eqref{eq.recall} that
\begin{align*}
\lim_{T\to\infty}\frac{n}T\cdot\frac1n\sum_{j=0}^{n-1}\hat\varphi( F_Y^j(x) ) &= \frac1{\int_{\Sigma'} R_Y\, d\mu_Y^i }  \int_{\Sigma'} \hat\varphi\, d\mu_Y^i\\
&=\frac1{\int_{\Sigma'} R_Y\, d\mu_Y^i }  \int_{\Sigma'}\int_0^{R_Y(x)}\varphi(Y_t(x)) dt d\mu_Y^i\\
&=\int_M \varphi d\nu_Y^i.
\end{align*}
Recalling~\eqref{eq.finale} and~\eqref{eq.chain}, we are left to show that 
\begin{equation}
\lim_{T\to\infty}\frac1T \int_{0}^{T_{n }(x)-T}\!\!\!\varphi(Y_t(F^n(x))) dt =0.
\end{equation}
Since
\begin{equation*}
\left|\frac1T \int_{0}^{T_{n }(x)-T}\!\!\!\varphi(Y_t(F^n(x))) \right|\le \frac{T_n(x)-T}T\sup|\varphi|\le \frac{T_n(x)-T_{n-1}(x)}T\sup|\varphi|
\end{equation*}
it is enough to prove that 
\begin{equation*}
\lim_{T\to\infty}\frac{T_n(x)-T_{n-1}(x)}T=0
\end{equation*}
Indeed, writing
\begin{equation*}
\frac{T_n(x)-T_{n-1}(x)}T= \frac{T_n(x)-T_{n-1}(x)}n\cdot\frac nT,
\end{equation*}
and recalling~\eqref{eq.recall}, the conclusion follows from~\eqref{eq.birkhoffR}.
\end{proof}

\section{Impulsive stability of Lorenz flows}
In this section we prove both the statistical stability and the entropy stability stated in Theorem~\ref{th.stability}. Recall that $F$, $R$ and $\mu$ are respectively the Poincar\'e map, the return time function and the (unique) $u$-Gibbs   measure for the Poincar\'e map $F$ associated with the Lorenz flow $X$. Recall also that~$\nu$ is the physical measure   for the Lorenz flow $X$ defined by the formula  in~\eqref{eq.ali0}.

Let $(\varphi_n)_n$ be a sequence of impulses converging to $\inc_\Sigma$ in the $C^1$ topology and   let   $\nu_n$  be   a  physical measure for the  semiflow $(Y^{n}_t)_t$ of the impulsive dynamical system $(M,X,\Sigma,\varphi_n)$. Let $t^-_n$, $t^+_n$, $\psi_n$ and $R_n$ be  functions respectively as in \eqref{eq.t-}, \eqref{eq.t+}, \eqref{eq.psi} and \eqref{eq.rushes} for each of these impulsive dynamical systems. 
Let also $\tilde F_n:\tilde\Sigma_n  \to\tilde\Sigma_n$ be related to $F_n:\Sigma'_n\to\Sigma'_n$ as in~\eqref{eq.lavinia}, where $\Sigma_n'=\varphi_n(\Sigma)$, $\tilde\Sigma_n=\psi_n(\Sigma_n')$ and $F_n$ is the Poincar\'e defined as in~\eqref{eq.leo} for the impulsive flow $Y_n$. 
By Corollary~\ref{co.finite2}, we have $\mu_n=(\psi_n^{-1})_*\tilde\mu_n$, where $ \mu_n$ is a $u$-Gibbs measure for $  F_n$ and $\tilde\mu_n$ is a $u$-Gibbs measure for $\tilde F_n$.
Let $\nu_n$ be one of the physical measures for the impulsive flow $Y_n$  related to $\mu_n$ by the formula~\eqref{eq.ali}; recall Proposition~\ref{pr.finite2}.

\begin{lemma}\label{le.S1to5}
  (S1)-(S5) hold for the sequence $(\tilde F_n)_n$ and $F$.
\end{lemma}

\begin{proof}
By Lemma~\ref{le.L1to5} we have  that (L1)-(L4) are satisfied for $F$ and therefore for all $F_n$ with uniform constants in (L3)-(L4), since~\eqref{eq.hadi} is valid and $\varphi_n,\psi_n$ can be taken arbitrarily close to the inclusion maps. Since the constants $A,\alpha,B,\beta,\lambda,\varepsilon_0, B_0$ and the cones $\mathcal K^u_x, \mathcal K^s_x$ appearing in (H1)-(H4) depend only on the constants in (L3)-(L4), then (S1) follows. For the other conditions note first $\Gamma_n=\Gamma$ and $D_{n,j}=D_j$, for $j=1,2$. Then conditions (S2), (S3) and (S5) are trivially satisfied and (S4) follows from~\eqref{eq.hadi}.
%
%[(S2)] The functions $f_j$ and $f_{n,j}$ have continuous extensions to the sets $D_j$ and $D_{n,j}$, respectively, for all $n$ and $j$.
%
%[(S3)] The sets $D_{n,j}$ converge to the sets  $D_j$, in the sense that for every $\varepsilon>0$ there is $n(\varepsilon)\in\N$ such that if $n>n(\varepsilon)$, then
%  $$
%  D_j\setminus \Gamma_\varepsilon\subset D_{n,j}\qand 
%  D_{n,j} \setminus \Gamma_{n,\varepsilon}\subset   D_j.
%  $$
% 
% [(S4)] The restriction of $f_n$ to  $D_{n,j}$ is an equicontinuous family and it converges to the restriction of $f$ to $D_j$ in the $C^0$ topology.
%
%[(S5)] On every  $D_{j}\setminus \Gamma_\varepsilon$ the sequence $f_n$  converges to   $f$   in the $C^1$ topology.
%
%
%
\end{proof}

\subsection{Statistical stability}
 
In this subsection we obtain the statistical stability part of Theorem~\ref{th.stability}. 
Since   $F$ has a unique $u$-Gibbs measure, it
 follows from Corollary~\ref{co.stable}, Lemma~\ref{le.L1to5} and Lemma~\ref{le.S1to5} that 
\begin{equation}\label{eq.jania}
\tilde\mu_n\stackrel{w*}\longrightarrow\mu,\quad\text{as $n\to\infty$.}
\end{equation}
Note that we may consider $\tilde\mu_n$ as a measure on the set $\Sigma$, since $\tilde\Sigma_n\subset \Sigma$.
We are left to  show that $\nu_n\to\nu$ in the weak$^\ast$ topology, as $n\to\infty$. 
Since the  measures $\nu$ and $\nu_n$ are given by  \eqref{eq.ali0} and \eqref{eq.ali} respectively, we just need to prove  that
\begin{enumerate}
\item $\displaystyle\int_{\Sigma'_n} R_{n}d\mu_n\to \int_{\Sigma} Rd\mu$,\quad as $n\to\infty$,
\end{enumerate}
and for all continuous $\phi:M\to \R$
\begin{enumerate}
\item[(2)] $\displaystyle\int_{\Sigma_n'}\int_0^{R_n(x)}\phi\circ Y^n_t(x) dt d\mu_n(x)
\to
 \int_{\Sigma}\int_0^{R(x)}\phi\circ X_t(x) dt d\mu(x)$,\quad as $n\to\infty$.
\end{enumerate}
This will be obtained in the   Lemma~\ref{lem11} and Lemma~\ref{lem02} below. In the next lemma we give an auxiliary result which will be used several times below. 
%We say that a double sequence $(a_{n,m})_{n,m}$ of real numbers converges to some $a\in\R$ if, for any $\varepsilon>0$, there exists $n_\varepsilon\in\N$ such that
%$$ m,n\ge n_\varepsilon \implies  |a_{n,m}-a|<\varepsilon.
%$$

\begin{lemma} \label{lem00}
 Assume that the following conditions hold:
\begin{enumerate}
\item[(c$_1$)]   there exists a sequence $(h_n)_n$ of continuous functions from $\Sigma$ to $\R$  converging  $\mu$-almost everywhere to      $h\in L^1( \mu )$, with  $h\in L^1(\tilde\mu_n)$, for all $n\ge1$;
\item[(c$_2$)] there exists    $g\in L^1(\mu)$  such that  $g\ge0$ and $|h_n|\le g$, for all $n\ge1$.
%\item[(C3)] 
%$\int_X (h-h_m)d\theta_n$ converges  to zero.
\end{enumerate}
Then 
$$\lim_{n\to\infty} \int_\Sigma h  d\tilde\mu_n = \int_\Sigma h  d\mu.
$$
\end{lemma}

\begin{proof}
 We may write for any $n,k\ge 1$
 $$
 \int_{ \Sigma} h  d\tilde\mu_n = \int_{ \Sigma } (h-h_k)  d \tilde\mu_n + \int_{ \Sigma } h_k  d\tilde\mu_n .
 $$
 Since the densities of the conditionals of each $\tilde\mu_n$ with respect to Lebesgue measure   on local unstable manifolds are  bounded from above and below by positive constants independent of $n$ (recall Remark~\ref{re.uniform}), we may assume that the conditionals of  $\tilde\mu_n$ are those of Lebesgue measure in the first integral of the last sum. Since the conditionals of   $ \mu$ with respect to Lebesgue measure   on local unstable manifolds are also bounded from above and below by  positive constants, it follows   that $h_k$ converges pointwise to $h$ for Lebesgue almost every point on local unstable manifolds. Therefore, $\int_{ \Sigma } (h-h_k)  d \tilde\mu_n$ converges to zero, by dominated convergence theorem.  
%By assumption,  the first term in the last sum is close to zero for   large~$m$ and  large $n$. 
Hence, taking first limit in $n$ and then limit in $k$ in the second term, the conclusion follows by the weak* convergence in~\eqref{eq.jania}    and the dominated convergence theorem. 
\end{proof}

%\begin{lemma} \label{lem00}
%Let $X$ be a compact metric space,  $(\theta_n)_n$ be  a sequence of probability measures on the Borel sets  of $X$ converging in the weak* topology to a probability $\theta_0$ on the Borel sets of $X$. Assume that the following conditions hold:
%\begin{enumerate}
%\item[(c$_1$)]   there exists a sequence $(h_n)_n$ of continuous functions from $X$ to $\R$  converging  $\theta_0$-almost everywhere to  a function   $h$ with  $h\in L^1(\theta_n)$, for all $n\ge0$;
%\item[(c$_2$)] there exists    $g\in L^1(\theta_0)$  with  $g\ge0$ and $|h_n|\le g$, for all $n\ge1$;
%\item[(C3)] 
%$\int_X (h-h_m)d\theta_n$ converges  to zero.
%\end{enumerate}
%Then 
%$$\lim_{n\to\infty} \int_X h  d\theta_n = \int_X h  d\theta_0.
%$$
%\end{lemma}
%
%\begin{proof}
% We may write
% $$
% \int_{ X} h  d\theta_n = \int_{X } (h-h_m)  d \theta_n + \int_{X } h_m  d\theta_n .
% $$
%By assumption,  the first term in the last sum is close to zero for   large~$m$ and  large $n$. 
%Hence, taking first limit in $n$ and then limit in $m$ in the last term, the conclusion follows by the weak* convergence of the sequence $(\theta_n)_n$ to $\theta_0$  and the dominated convergence theorem. 
%\end{proof}

\begin{lemma}\label{lem11}
$\displaystyle\int_{\Sigma'_n} R_{n}d\mu_n\to \int_{\Sigma} Rd\mu $, as $n\to\infty$.
\end{lemma}

\begin{proof}
It follows from \eqref{eq.rushes} that for all $n$
$$
R_n = R\circ\psi_n +t^-_n ,
$$
which together with Corollary~\ref{co.finite2} yields
$$
\int_{\Sigma'_n} R_{n}d\mu_n= \int_{\Sigma'_n} (R\circ\psi_n +t^-_n )d(\psi_n^{-1})_*\tilde\mu_n 
=
 \int_{\tilde\Sigma_n} (R +t^-_n\circ\psi_n^{-1} )d\tilde\mu_n.
$$
Since $t^-_n\circ\psi_n^{-1}$ is uniformly small when $\varphi_n$ converges to $\inc_\Sigma$ in the $C^0$ topology and the support of $\tilde\mu_n$ is   a subset of $\Sigma$, we just need to prove that
 \begin{equation}\label{eq.me}
\int_{ \Sigma } R  d\tilde\mu_n\to \int_{ \Sigma } R  d \mu,\quad\text{as $n\to\infty$.} 
\end{equation}
 Set for each $k\ge1$
  $$ \hat R_k=\min\{ R, k\}.
  $$
We conclude the proof using Lemma~\ref{lem00} with $h_k=\hat R_k$ and $h=g=R$. 
The conditions (c$_1$) and (c$_2$) are clearly verified, by Lemma~\ref{lem01} and recalling that $R$ is integrable with respect to the measure $\mu$.
%Moreover, since the densities of the conditionals of each $\tilde\mu_n$ with respect to Lebesgue measure   on local unstable manifolds are  bounded from above and below by uniform constants independent of $n$ (recall Remark~\ref{re.uniform}), we may replace $\tilde\mu_n$ by Lebesgue measure in the integral
% $\int_\Sigma (R-\hat R_m)d\tilde\mu_n.$
%Noting that $\hat R_m$ converges pointwise to $R$, we obtain (C3) by dominated convergence theorem.  
\end{proof}

 \begin{lemma}\label{lem02}
%\begin{equation}\label{555}
For  any continuous $\phi:M\to \R$, we have
$$\lim_{n\to\infty}\int_{\Sigma_n'}\int_0^{R_n(x)}\phi\circ Y^n_t(x) dt d\mu_n(x)
=
 \int_{\Sigma}\int_0^{R(x)}\phi\circ X_t(x) dt d\mu(x). $$%, \quad\text{as $n\to\infty$.}$$
\end{lemma}
\begin{proof}
We have
\begin{align*}
\int_{\Sigma_n'}\int_0^{R_n(x)}\phi\circ Y^n_t(x) dt d\mu_n(x) &= 
\int_{\Sigma_n'}\int_0^{R\circ\psi_n(x) +t^-_n(x)}\phi\circ Y^n_t(x) dt d(\psi_n^{-1})_*\tilde\mu_n(x)\\
&= 
\int_{\tilde\Sigma_n}\int_0^{R(x) +t^-_n\circ\psi_n^{-1}(x)}\phi\circ Y^n_t\circ \psi_n^{-1}(x) dt d \tilde\mu_n(x)
\end{align*}
Since for any $x\in\tilde\Sigma_n$
\begin{equation*}
 \int_0^{R(x) +t^-_n\circ\psi_n^{-1}(x)}\!\!\!\!\phi\circ Y^n_t\circ \psi_n^{-1}(x) dt  =
 \int_0^{R(x)  }\!\!\!\!\phi\circ Y^n_t\circ   \psi_n^{-1}(x) dt  
+
 \int_0^{ t^-_n\circ\psi_n^{-1}(x)}\!\!\!\!\phi\circ Y^n_t\circ \psi_n^{-1}(x) dt  
\end{equation*}
and    $t^-_n\circ\psi_n^{-1}$ can be made uniformly small when $\varphi_n$ converges to $\inc_\Sigma$ in the $C^0$ topology, we just need to prove that
\begin{equation*}
\int_{ \Sigma }\int_0^{R(x) }\phi\circ Y^n_t\circ \psi_n^{-1}(x) dt d \tilde\mu_n(x)
\to
\int_{ \Sigma }\int_0^{R(x)  }\phi\circ X_t  (x) dt d  \mu (x), \quad\text{when $n\to\infty$}.
\end{equation*}
Note that we may assume that $\tilde\mu_n$ is  a measure on $\Sigma$.
We have
\begin{align}
\int_{ \Sigma }\int_0^{R(x) }&\phi\circ Y^n_t\circ \psi_n^{-1}(x) dt d \tilde\mu_n(x)
-
\int_{ \Sigma }\int_0^{R(x)  }\phi\circ X_t  (x) dt d  \mu (x)\nonumber\\
 = &
\int_{ \Sigma }\int_0^{R(x) } (\phi\circ Y^n_t\circ \psi_n^{-1}(x) - \phi\circ X_t  (x) )dt d \tilde\mu_n(x)\label{eq.first}\\
&+\int_{ \Sigma }\int_0^{R(x)  }\phi\circ X_t  (x) dt d  \tilde\mu_n (x)-\int_{ \Sigma }\int_0^{R(x)  }\phi\circ X_t  (x) dt d  \mu (x).\label{eq.last}
\end{align}
First we prove that the difference in \eqref{eq.last} converges to 0 when $n\to\infty$. Indeed, set
$$
h(x)=\int_0^{R(x)  }\phi\circ X_t  (x) dt  
  $$
and for $k\ge1$,
  $$\hat R_k(x)=\min\{ R(x), k\}\qand h_k(x)=\int_0^{\hat R_k(x)  }\phi\circ X_t  (x) dt  .
  $$
  We have
 $$|h_k(x)|=\left| \int_0^{\hat R_k(x)  }\phi\circ X_t  (x) dt \right| \le  \int_0^{\hat R_k(x)  }\left|\phi\circ X_t  (x)  \right| dt \le C\hat R_k(x)\le CR(x).
 $$
We conclude the proof using Lemma~\ref{lem00} with     $g=CR$. 
The conditions (c$_1$) and (c$_2$) are clearly verified, by Lemma~\ref{lem01} and recalling that $R$ is integrable with respect to $\mu$.

We are left to prove that the term in \eqref{eq.first} converges to 0 when $n\to\infty$. For simplicity, consider 
$$
g(t,x)=\phi\circ X_t  (x) \qand g_n(t,x)=\phi\circ Y^n_t\circ \psi_n^{-1}(x), \quad\forall n\ge1.
$$
It is enough to show that 
$$
\int_{ \Sigma }\int_0^{R(x) } |g_n(t,x) - g  (t,x) |dt d \tilde\mu_n(x)\to 0
$$
when $n\to\infty$.
Since $\hat R_k(x)\to R(x)$, when $k\to\infty$, for $\tilde\mu_n$-almost every $x$, by monotone convergence theorem
 \begin{align*}
\int_{ \Sigma }\int_0^{R(x) } |g_n(t,x) - g  (t,x) |dt d \tilde\mu_n(x)
&=\int_{ \Sigma }\lim_{k\to\infty} \int_0^{\hat R_k(x) } |g_n(t,x) - g  (t,x) |dt d \tilde\mu_n(x)\nonumber\\
&=\lim_{k\to\infty} \int_{ \Sigma } \int_0^{\hat R_k(x) } |g_n(t,x) - g  (t,x) |dt d \tilde\mu_n(x). \label{eq.wael}
\end{align*}
Since this last expression defines an increasing sequence in $k$ of nonnegative numbers, we just need to show that each term in that sequence can be made arbitrarily small (for~$n$ sufficiently large).
So, fix  an arbitrary    $\varepsilon>0$.  Given any   $k\in\N$,   there exists $n_\varepsilon\in \N $ (depending only on $k$) such that
for all   $x\in \Sigma$  and $0\le t\le \hat R_k(x)$ (recall that $\hat R_k(x)\le k$), we have 
\begin{equation*}
n\ge n_\varepsilon \implies |g_n(t,x)- g (t,x)|<\varepsilon.
\end{equation*}
%Recall that $\hat R_k(x)\le m$, for all $x\in\Sigma$.
Therefore,   for any $n\ge n_\varepsilon$ we have
\begin{align*}
\int_{ \Sigma } \int_0^{\hat R_k(x) } |g_n(t,x) - g  (t,x) |dt d \tilde\mu_n(x) &\le \varepsilon\int_\Sigma \hat R_k(x)d \tilde\mu_n(x) \le \varepsilon\int_\Sigma R (x)d \tilde\mu_n(x).
\end{align*}
Now, by~\eqref{eq.me}, we may choose a uniform constant $C>0$ such that $\int_\Sigma R (x)d \tilde\mu_n(x)\le C$, for all $n\in\N$. This finishes the proof.
\end{proof}
\subsection{Entropy stability}

In this subsection we obtain  the entropy stability part of Theorem~\ref{th.stability}. 
%
%Let $(\varphi_n)_n$ be a sequence of impulses converging to $\inc_\Sigma$ in the $C^1$ topology, let       $\nu_n$  be  a  physical measure for the impulsive semiflow $(Y^{n}_t)_t$ of the impulsive dynamical system $(M,X,\Sigma,\varphi_n)$ and let   $R_n$ be the function associated with $Y^{n}$ as in   \eqref{eq.rushes}. Let also $\tilde F_n:\tilde\Sigma_n  \to\tilde\Sigma_n$ be related to $F_n$ as in~\eqref{eq.lavinia}, with $\tilde\Sigma_n$ as before.
%
%
Similar to \eqref{eq.abramov}, we have by Abramov formula 
\begin{equation}\label{eq.abramovy}
h_{\nu_n}(Y_n)=\frac{h_{\mu_n}(F_n)}{\int_{\Sigma'_n} R_n d\mu_n } .
\end{equation}
%
%
%
%
%Recall that $F$, $R$ and $\mu$ are respectively the Poincar\'e map, the return time and the $u$-Gibbs   measure for $F$ associated with the Lorenz flow $X$ introduced in Subsection~\ref{se.Pmap}.
%Let also $\nu$ be the physical measure   for the Lorenz flow $X$ introduced  in~\eqref{eq.ali0}.
%By Corollary~\ref{co.finite2}, we know that $\mu_n=(\psi_n^{-1})_*\tilde\mu_n$, where $\tilde\mu_n$ is a $u$-Gibbs measure for $\tilde F_n$.
%It follows from \cite[Theorem~7.2]{S92} that 
%\begin{equation}\label{eq.jania2}
%\tilde\mu_n\stackrel{w*}\longrightarrow\mu,\quad\text{as $n\to\infty$.}
%\end{equation}
%Note that we may consider $\tilde\mu_n$ as a measure on $\Sigma$, since $\tilde\Sigma_n\subset \Sigma$.
By \eqref{eq.abramov} and \eqref{eq.abramovy}, it is enough   to  prove that 
\begin{equation*}
\frac{h_{\mu_n}(F_n)}{\int_{\Sigma'_n} R_n d\mu_n }\longrightarrow  \frac{h_{\mu}(F)}{\int_{\Sigma} R d\mu },
\quad\text{as $n\to\infty$} .
\end{equation*}
Since $(F_n,\mu_n)$ and $(\tilde F_n,\tilde\mu_n)$ are isomorphic dynamical systems, by Lemma~\ref{lem11}   we just need to prove that
\begin{equation}\label{eq.goal}
 h_{\tilde\mu_n}(\tilde F_n) \to   h_{\mu}(F),
\quad\text{as $n\to\infty$} .
\end{equation}
It follows from Theorem~\ref{th.finite},  Theorem~\ref{th.L->H}  and Lemma~\Ref{le.L1to5} that the entropy formula holds for both  dynamical systems $(\tilde F_n,\tilde\mu_n)$ and $( F , \mu )$. Therefore,  the  convergence in~\eqref{eq.goal} can be rephrased as
\begin{equation}\label{eq.newgoal}
 \int_\Sigma \log |\tilde J^u_n| d\tilde\mu_n \to  \int_\Sigma \log|J^u|  d\mu ,
\quad\text{as $n\to\infty$} ,
\end{equation}
where $J_n^u$ and $J^u$ stand for the unstable Jacobians of $\tilde F_n$ and $F$, respectively.
%\bibliographystyle{abbrv}
%\bibliography{/Users/jfalves/Mega/Bibdesk/Biblio}

\begin{lemma}\label{le.entropy}
$\displaystyle\lim_{n\to\infty}\int_{\Sigma}\log |\tilde J^u_n| d\tilde\mu_n= \int_{\Sigma}  \log|J^u|  d\mu $.
\end{lemma}

\begin{proof} 
We have
\begin{align}
\int_{\Sigma}\log |\tilde J^u_n| d\tilde\mu_n - &\int_{\Sigma}  \log|J^u| d\mu\nonumber\\
 & =
\int_{\Sigma}\left(\log |\tilde J^u_n|  -  \log|J^u| \right)d\tilde\mu_n +
 \int_{\Sigma}  \log|J^u| d\tilde\mu_n- \int_{\Sigma}  \log|J^u| d\mu .\label{eq.difference}
\end{align}
We first prove that the integral of the difference above converges to zero when $n\to\infty$. 
%It is enough to show that 
%$$
%\int_{ \Sigma }  |\log |\tilde J^u_n| - \log |  J^u | |  d \tilde\mu_n(x)\to 0
%$$
%when $n\to\infty$. 
By Hadamard's inequality and (H1), we have for all $x$
\begin{equation}\label{eq.Hadamard}
|\tilde J^u_n(x)|\le \| d_x F_n\|\le  A \rho(x,\Gamma)^{-\alpha};
\end{equation}
recall that the constant $A$ may be chosen uniform, by Lemma~\ref{le.S1to5}.
Since the densities of the conditionals of each $\tilde\mu_n$ with respect to Lebesgue measure   on local unstable manifolds are  bounded from above and below by positive constants independent of $n$ (recall Remark~\ref{re.uniform}), we may assume that the conditionals of  $\tilde\mu_n$ are those of Lebesgue measure.
Moreover, it follows from (S5)   that $\log |\tilde J^u_n|$ converges pointwise to $\log |\tilde J^u|$, Lebesgue almost everywhere. Since $\log \rho(x,\Gamma)^{-\alpha}$ is integrable with respect to Lebesgue measure, using \eqref{eq.Hadamard} and dominated convergence theorem we conclude that the integral of the difference in~\eqref{eq.difference} converges to zero when $n\to\infty$.

We are left to show that the difference of the integrals in~\eqref{eq.difference} goes to zero when~${n\to\infty}$. Decomposing both integrals into    integrals on $D_1$ and $D_2$, and noticing that $\log |  J^u | $ is continuous on both $D_1$ and $D_2$, the conclusion follows by the weak* convergence of $\tilde\mu_n$ to~$\mu$ given by~\eqref{eq.jania}.
\end{proof}

%\bibliographystyle{abbrv}
%\bibliography{/Users/jfalves/Mega/Bibdesk/Biblio}

\end{document}